\documentclass[12pt,twoside, draft]{amsart}

\usepackage{amsmath,amssymb,amsthm,enumerate,color}
\usepackage{orcidlink}
\usepackage[left=2.5cm,top=3cm,right=2.5cm,bottom=3cm]{geometry}

\theoremstyle{plain}
\newtheorem{thm}{Theorem}[section]
\newtheorem{lem}[thm]{Lemma}

\newtheorem{defi}[thm]{Definition}

\newtheorem{rem}[thm]{Remark}

\newcommand{\set}[1]{\left\{#1\right\}}
\newcommand{\R}{\text{$\mathbb{R}$}}

\newcommand{\di}{\displaystyle}

\begin{document}

\title[Uniqueness of Best Approximation]{On the uniqueness of best approximation in Orlicz spaces}

\author[A. Benavente]{Ana Benavente$^1$ \orcidlink{0000-0003-2351-1746}}
\address{$^1$ Instituto de Matem\'{a}tica Aplicada San Luis, UNSL-CONICET and
Departamento de Mate\-m\'{a}tica, FCFMyN, UNSL, Av. Ej\'{e}rcito de los
Andes 950, 5700 San Luis, Argentina.} \email{abenaven@unsl.edu.ar}

\author[J. Costa Ponce]{Juan Costa Ponce$^2$}
\address{$^2$ Instituto de Matem\'{a}tica Aplicada San Luis, UNSL-CONICET and
Departamento de Mate\-m\'{a}tica, FCFMyN, UNSL, Av. Ej\'{e}rcito de los
Andes 950, 5700 San Luis, Argentina.}
\email{costaponcejuan@gmail.com}

\author[S. Favier]{Sergio Favier$^3$ \orcidlink{0000-0002-2395-7652}}
\address{$^3$ Instituto de Matem\'{a}tica Aplicada San Luis, UNSL-CONICET and
Departamento de Mate\-m\'{a}tica, FCFMyN, UNSL, Av. Ej\'{e}rcito de los
Andes 950, 5700 San Luis, Argentina}
\email{sfavier@unsl.edu.ar}


\thanks{Research partially supported CONICET, Universidad Nacional de San Luis.}%



\begin{abstract}
We study uniqueness of best approximation in Orlicz spaces $L^{\Phi},$ for  different types  of convex functions $\Phi$ and for some finite dimensional approximation classes of functions, where Tchebycheff spaces, and more general approximation ones, are involved. 

\end{abstract}

\maketitle


\section{Introduction and notations}

Let $\Im$ be the class of all  non decreasing functions $\varphi$
defined for all real numbers $t\geq 0,$ such that $\varphi(t)>0$ if $t>0,$  and consider $\Psi$ the class of all convex functions $\Phi$ defined by $\Phi (x) = \int_0^x \varphi (t) \, dt$, for $x\geq 0$, with $\varphi \in \Im$. We  assume a $\Delta_2$ condition for
the functions $\Phi,$ which means that there exists a constant
$\Lambda>0$ such that $\Phi(2 x)\leq \Lambda
\Phi(x),$ for $x\geq 0.$ We also denote by  $\varphi^-$ and  $\varphi^+$  the left and right derivatives of $\Phi$ respectively.

Let $\Phi\in\Psi$ and let $(\Omega,\mathcal{A},m)$ be the Lebesgue
measure space where $\Omega \subset \mathbb{R}$ is a bounded set. We
denote by $L^{\Phi}=L^{\Phi}(\Omega,\mathcal{A},m)$  the Orlicz space given by the class of all
$\mathcal{A}$-measurable functions $f$ defined on $\Omega$ such that
$\int_{\Omega}\Phi(|f|)dx<\infty.$

Given a set $S\subset L^\Phi$, an element $P\in S$ is called
{\em a best $\Phi$-approximation of $f\in L^\Phi$ from the
approximation class $S$} if and only if
$$\int_{\Omega}\Phi(|f-P|)dx=\inf_{Q\in S}\int_{\Omega}\Phi(|f-Q|)dx$$ 
and, in this case, we write $P\in\mu_{\Phi}(f/ S) .$ The
mapping $\mu_\Phi: L^\Phi \to S$ is called the {\em best
$\Phi$-approximation operator} given $S$.

Given a $n-$dimensional linear space $S_n \subseteq C([a,b])$, the existence of a best approximation is assured by Theorem 2.3 in \cite{BFL}.
Recall that $S_n$ is called a Tchebycheff space  if any non zero element in $S_n$ has at most $n-1$ zeros in $[a,b].$

Uniqueness of best approximation  was studied since 1924, for example in \cite{Jackson}, and this problem was extensively developed in $L^1$ in the 70´s. Conditions which assure uniqueness of best $L^1$ approximation have been studied, for instance for  Galkin in \cite{Galkin} and Strauss \cite{Strauss} showed uniqueness in $L^1$ by polynomial splines. In \cite{Micchelli} Micchelli considered best approximation by weak Tchebychev subspaces. For a very mentioned reference of this problem we also have  to cite to DeVore \cite{DeVore}, Kroo \cite{Kroo}, Pinkus \cite{Pinkus} and Strauss \cite{Strauss}. 

We point out that for Orlicz spaces $L^{\Phi} ([a,b])$ it is easy to obtain uniqueness of best $\Phi-$ approximation for a given strictly convex function $\Phi.$ In this paper we present first a characterization theorem of best $\Phi-$ 
 approximation which will be used in the sequel. Next,   we get uniqueness of the best $\Phi-$ approximation,  for a continuous function defined in a real interval $[a,b],$ where we consider a general function $\Phi$  and where the approximation class is a Tchebycheff space which generalizes the classical result in $L^1,$ setted for example  in \cite{Rice}.  Next we get uniqueness of the best $\Phi -$ approximation for more general approximation classes, considering suitable classes of functions $\Phi.$  We set also a sufficient and necessary condition to assure that the best approximation of a continuous function is unique, which is an extension, in some way, to Theorem 22 of \cite{CW}, stablished  in the $L^1$ space. 


\section{Characterization and properties of the best $\Phi-$approximations.}

In this section, we will set a characterization  property of the  best $\Phi-$ approximations, when the approximation class is a finite dimensional vector space $S_n,$ which will be useful to get uniqueness of best $\Phi-$ approximation for continuous functions. On the other hand, it is well known that uniqueness results follows straight forward  if $\Phi$ is a strictly convex function. We deal with   a more general convex functions $\Phi,$ where the  strictly convexity is not required. At the end of this section, we will present a property for the not strictly convex functions that will be used in the study of the uniqueness  for a suitable class of approximation functions.

The following  characterization is a generalization of Theorem 2.2 in \cite{AFL} for the case $\varphi^{+} (0)>0$ and Theorem 2.2 in \cite{Acinas Favier 2016} for the non derivative case of $\Phi.$ 

\begin{thm}\label{thm: caracterización}
    Let $\Phi\in \Psi$ and $f\in L^{\Phi}(\Omega)$. Then  $P\in \mu_{\Phi}(f/S_n)$ if and only if
        \begin{equation*}\label{eq:caract-bpa-f>P}
          \begin{split}
            \di\int_{\{Q>0,~ f>P\}\cup\{Q<0,~ f < P\}} \varphi^{-}(|f-P|)|Q| \,dx& -  \int_{\{Q<0, ~ f>P \}\cup\{Q>0,~ f < P\}} \varphi^{+}(|f-P|)|Q| \,dx \\
                                                                                 & \leq \varphi^{+}(0)\int_{\{ f=P \}} |Q|dx, 
          \end{split}
        \end{equation*}
for any $Q\in S_n$.
\end{thm}
\begin{proof}
    For  $P\in \mu_{\Phi}(f/S_n)$,~ $Q\in S_n$ with $Q\not=P$ and $\epsilon\geq 0$, we define $$F_Q(\epsilon)=\int_\Omega  \Phi(|f-(P+\epsilon Q)|)~dx.$$
    Using the convexity of $\Phi,$ we have that $F_Q$ is also a convex function on $[0,\infty)$, then
    \begin{equation*}
     \begin{split}
     F_{Q}(a \epsilon_{1}+b \epsilon_{2})
     &=\int_{\Omega} \Phi\left(\left|(a+b) f-\left((a+b) P+\left(a \epsilon_{1}+b \epsilon_{2}\right) Q\right)\right|\right) dx \\
     & \leq  \di\int_{\Omega} \Phi\left(a\left|f-\left(P+\epsilon_{1} Q\right)\right|+b\left|f-\left(P+\epsilon_{2} Q\right)\right|\right) \,dx \\
     & \leq a \int_{\Omega} \Phi\left(\left|f-\left(P+\epsilon_{1} Q\right)\right|\right) d x+b \int_{\Omega} \Phi\left(\left|f-\left(P+\epsilon_{2} Q\right)\right|\right) \,dx \\
     &= a F_{Q}\left(\epsilon_{1}\right)+b F_{Q}\left(\epsilon_{2}\right),
     \end{split}
   \end{equation*}
for $\epsilon_1,~\epsilon_2\geq 0$, and $a+b=1.$
It follows that
$F_{Q}(0)=\min\limits _{[0, \infty)} F_{Q}(\epsilon),$
and the equality holds if, and only if, $0 \leqslant F_{Q}^+(0)$, where 
$F_{Q}^+(0)$ is the right derivative of $F_{Q}$ in $\epsilon=0$. Now we  compute the derivative.
\begin{equation*}\label{eq:derivative-F-at-0}
\begin{aligned}
F_{Q}^+(0) &=\lim _{\epsilon \rightarrow 0^{+}} \frac{F_{Q}(0+\epsilon)-F_{Q}(0)}{\epsilon} \\
&=\lim _{\epsilon \rightarrow 0^{+}} \frac{1}{\epsilon}\left\{\int_{\Omega} \Phi(|f-(P+\epsilon Q)|) \, dx-\int_{\Omega} \Phi(|f-P|) \, dx\right\}\\
&=\lim _{\epsilon \rightarrow 0^{+}} \int_{\Omega\cap \set{Q\not =0}} \left[ \frac{\Phi(|f-(P+\epsilon Q)|) - \Phi(|f-P|)}{\epsilon Q}\right]Q \,dx.
\end{aligned}
\end{equation*}
To analyze each absolute value in the argument of $\Phi,$ we split the set $\Omega\cap \set{Q\not =0}$ in seven cases, which will yield the following seven integrals:
\begin{equation}\label{eq:F_Q derivative on zero}
\begin{aligned}
F_{Q}^+(0)
=\lim\limits_{\epsilon \rightarrow 0^{+}}
&\left\{\int_{\{Q>0,~P<P+\epsilon Q<f\}}\left[\frac{\Phi(f-P-\epsilon Q)-\Phi(f-P)}{\epsilon Q}\right] Q \, dx\right. \\
&+\int_{\{Q<0,~P+\epsilon Q<P<f\} }\left[\frac{\Phi(f-P-\epsilon Q)-\Phi(f-P)}{\epsilon Q}\right] Q \, dx \\
&+\int_{\{Q<0,~ P+\epsilon Q<f<P\}}\left[\frac{\Phi(f-P-\epsilon Q)-\Phi(P-f)}{\epsilon Q}\right] Q \, dx \\
&+\int_{\{Q<0,~ P+\epsilon Q<f=P\}}\left[\frac{\Phi(f-P-\epsilon Q)-\Phi(P-f)}{\epsilon Q}\right] Q \, dx \\
&+\int_{\{Q>0,~ f \leqslant P < P+\epsilon Q \} }\left[\frac{\Phi(P+\epsilon Q-f)-\Phi(P-f)}{\epsilon Q}\right] Q \, dx \\
&+\int_{\{Q>0,~ P< f \leqslant P+\epsilon Q\}}\left[\frac{\Phi(P+\epsilon Q-f)-\Phi(f-P)}{\epsilon Q}\right] Q \, dx \\
&\left.+\int_{\{Q<0,~ f \leqslant P+\epsilon Q < P\} }\left[\frac{\Phi(P+\epsilon Q-f)-\Phi(P-f)}{\epsilon Q}\right] Q \,dx\right\},
\end{aligned}
\end{equation}
that is,
$F_{Q}^+(0)=
\lim\limits_{\epsilon \rightarrow 0^{+}} \left\{I_1+I_2+I_3+I_4+I_5+I_6+I_7\right\}$.

By the convexity of $\Phi$, for  $0 < \epsilon \leqslant 1$  we have:
\begin{equation}\label{acotacion}
\begin{split}
    \frac{|Q||\Phi(|f-(P+\epsilon Q)|)- \Phi(|f-P|)|}{\epsilon |Q| }
& \leqslant |Q| (\varphi^{+}(|f-P|+|Q|)+\varphi^{+}(|f-P|)).
\end{split}
\end{equation}
The function to the right of the inequality (\ref{acotacion}), is integrable in $\Omega$  because $\varphi^+ \in \Delta_2$, $|P|\leqslant \|P\|_{\infty}$,  $|Q|\leq \|Q\|_{\infty}$ and $|\Omega|<\infty$. By the Dominated Convergence theorem we conclude that:
\begin{equation*}\label{eq:limI_1}
\lim\limits_{\epsilon \rightarrow 0^{+}} I_1=
-\int_{\{Q>0,~ f > P\} } \varphi^{-}(|f-P|) Q \,dx,
\end{equation*}

\begin{equation*}\label{eq:limI_2}
\lim\limits_{\epsilon \rightarrow 0^{+}} I_2=
-\int_{\{Q<0, ~ f > P\}} \varphi^{+}(|f-P|) Q \,dx,
\end{equation*}

\begin{equation*}\label{eq:limI_4}
\begin{split}
\lim\limits_{\epsilon \rightarrow 0^{+}} I_4&
= \lim\limits_{\epsilon \rightarrow 0^{+}}\int_{\{Q<0,~ P+\epsilon Q<f=P\}}\left[\frac{\Phi(0)- \Phi(0 +\epsilon |Q|)}{\epsilon |Q|}\right] Q \,dx\\
&=\varphi^{+}(0)\int_{\{Q<0,~ f=P\}} |Q|dx,
\end{split}
\end{equation*}

\begin{equation*}\label{eq:limI_5}
\begin{split}
\lim\limits_{\epsilon \rightarrow 0^{+}} I_5 &=
 \int_{\{Q>0,~ f \leqslant P\} } \varphi^{+}(|f-P|) Q \,dx\,\\
 &= \int_{\{Q>0,~ f < P\}} \varphi^{+}(|f-P|) Q \,dx\,
 +\,\varphi^{+}(0)\int_{\{Q>0,~ f=P\} }|Q|dx
\end{split}
\end{equation*}

and
\begin{equation*}\label{eq:limI_7}
\lim\limits_{\epsilon \rightarrow 0^{+}} I_7=
\int_{\{Q<0,~ f < P\}} \varphi^{-}(|f-P|) Q \,dx.
\end{equation*}

For the remaining integrals, let us consider  a decreasing and convergent sequence $\left\{\epsilon_{n}\right\}\searrow 0$, and the decreasing sets:
$A_{n}:=\left\{Q<0,~ P+\epsilon_{n} Q<f<P\right\}$,
 then $\bigcap\limits_{n=1}^{\infty} A_{n}= \emptyset$ and:
\begin{equation*}
\begin{split}
 \lim\limits_{n\to\infty} \int_{A_{n}}\left[\frac{\Phi(f-P-\epsilon_n Q)-\Phi(P-f)}{\epsilon_{n} Q}\right] Q dx &\le
\lim\limits_{n\to\infty} \int_{A_n}  \varphi^{+}(|f-P|+|Q|)+\varphi^{+}(|f-P|) dx\\
&= \lim\limits_{n\to\infty} \mu\left(A_{n}\right),
\end{split}
\end{equation*}
and
$\lim\limits_{n \rightarrow \infty} \mu\left(A_{n}\right) = \mu\left(\bigcap\limits_{n=1}^{\infty} A_{n}\right)= \mu(\emptyset)
=0$. Then
$\lim\limits_{\epsilon \rightarrow 0^{+}} I_3=0.$

Now, let us consider the sets $B_{n}=\left\{Q>0,~ P< f \leqslant P+\epsilon_{n} Q\right\} $,
then
$\bigcap\limits_{n=1}^{\infty}B_n=\emptyset$ and again, by the Dominated Convergence Theorem we have
$\lim\limits_{\epsilon \rightarrow 0^{+}} I_6=0$.

Finally, replacing all the results in (\ref{eq:F_Q derivative on zero}), we obtain
\begin{equation*}
\begin{split}
F_{Q}^{+}(0)=&-\int_{\{Q>0,~ f > P\} } \varphi^{-}(|f-P|) Q \,dx -\int_{\{Q<0, ~ f > P\}} \varphi^{+}(|f-P|) Q \,dx, \\
&+\int_{\{Q>0,~ f < P\}} \varphi^{+}(|f-P|) Q \,dx\,
 +\,\varphi^{+}(0)\int_{\{f=P\} }|Q|dx \\
 &+\int_{\{Q<0,~ f < P\}} \varphi^{-}(|f-P|) Q \,dx.
\end{split}
\end{equation*}

As $F_{Q}^{+}(0) \geqslant 0$, we conclude the statement of the theorem.
\end{proof}

For the next result we use the following notations. For a real number $x$ we set $sgn(x) =\frac{x}{|x|},\,\, x\neq 0,$ and $sign(0) = 0.$

\begin{rem}\label{remark caracterizacion}
 If $\Phi (x) = \int_0^x \varphi (t) \, dt$ is a derivable function  then $\varphi^+=\varphi^-=\varphi$, thus  the last result becomes: $P\in\mu_{\Phi}(f/ S_n)$ if and only if
\begin{equation*}
    \di\left|\int_\Omega \varphi(|f-P|) sgn(f-P)Q \,dx\right| \leq \varphi(0^+)\int_{\{ f=P \}} |Q|dx,  
\end{equation*}

for each $Q\in S_n.$

\end{rem}

The following result generalizes the Haar's $L_1-$version theorem in \cite{Jackson}.

\begin{lem}\label{lema 2 mejores app}
    Let $\Phi\in\Psi$ and $S_n$ be an $n-$dimensional linear space of continuous functions in almost every point in the interval $[a,b]$. 
    If $P_1, P_2 \in \mu_{\Phi}(f/S_n)$  for $f\in L^\Phi$, then
    $$[f(x)-P_1(x)][f(x)-P_2(x)]\geq 0$$ for almost every $x\in[a,b]$.
\end{lem}
\begin{proof}
    Let's define the functions $g_1(x):=f(x)-P_1(x)$, $g_2(x):=f(x)-P_2(x)$ and the number $\rho:=\inf_{Q\in S_n}\int_a^b \Phi(|f-Q|)dx$. We observe that $\rho=\int_a^b \Phi(|g_1|)dx=\int_a^b \Phi(|g_2|)dx$.
    Using that $\Phi$ is a non decreasing and convex function, we have: $$\int_a^b\Phi(|f-\frac{P_1+P_2}{2}|dx\leq \frac{1}{2}\int_a^b \Phi(|g_1|)dx + \frac{1}{2}\int_a^b \Phi(|g_2|)dx=\rho.$$
    Then $\frac{P_1+P_2}{2}$ is also a best $\Phi-$approximation for $f$ and $$\int_a^b \Phi\left(\frac{|g_1+g_2|}{2}\right)dx=\int_a^b \frac{\Phi(|g_1|)}{2}dx + \int_a^b \frac{\Phi(|g_2|)}{2}dx.$$ Which means that  
    \begin{equation}\label{eq Phi}
        \Phi\left(\frac{|g_1(x)+g_2(x)|}{2}\right)=\frac{\Phi(|g_1(x)|)}{2}+\frac{\Phi(|g_2(x)|)}{2}
    \end{equation}in almost every $x\in [a,b]$.
    
    That is because $\Phi$ is continuous and $\frac{\Phi(|g_1(x)|)}{2} + \frac{\Phi(|g_2(x)|)}{2} -\Phi\left(\frac{|g_1(x)+g_2(x)|}{2}\right)\geq 0$ in $[a,b]$. If for some set $A$ with positive measure, $g_1(x)$ and $g_2(x)$ have different signs, then $|g_1(x)+g_2(x)|<|g_1(x)|+|g_2(x)|$ in $A$, and since $\Phi$ is a increasing and convex function, it satisfies $\Phi(\frac{|g_1(x)+g_2(x)|}{2})< \Phi(\frac{|g_1(x)|+|g_2(x)|}{2})\leq \frac{\Phi(|g_1(x)|+\Phi(|g_2(x)|)}{2}$. But this last inequality contradicts (\ref{eq Phi}), and the proof is then complete.
    
\end{proof}


Recall that a function $\Phi:[0,\infty)\to\R$ is {\em convex} in an interval $I$ if for all $x,y\in I$ and $0\leq \lambda\leq 1$, then  $\Phi(\lambda x+(1-\lambda)y)\leq \lambda \Phi(x)+(1-\lambda)\Phi(y)$. And  $\Phi$ is a {\em strictly convex} function in $I$ if for all $x\not=y\in I$ and $0< \lambda < 1$, then  $\Phi(\lambda x+(1-\lambda)y)<\lambda \Phi(x)+(1-\lambda)\Phi(y)$. We observe that linear functions are convex, but not strictly convex.
On the other hand, the following lemma shows that a convex, but not strictly convex function, is linear  in some sub interval. 

\begin{lem}\label{lema no estrict conv}
    Let $\Phi\in \Psi$. If $\Phi$ is not strictly convex in an interval $I$, then there exists an interval $J\subset I$ such that $\Phi$ is a straight line in $J$.
\end{lem}
\begin{proof}
       If $\Phi$ is not strictly convex, then there exists an interval $J=(x_1,x_3)\subset I$ and a convex combination of $x_1$ and $x_3$ called $x_2$, such that the point $(x_2,\Phi(x_2))$ is in the line between the points $(x_1,\Phi(x_1))$ and $(x_3, \Phi(x_3))$, where the equation for this line is  
       %
       \begin{equation}\label{eq de la cuerda cuerda}
           \Phi(x)=\frac{\Phi(x_3)-\Phi(x_2)}{x_3-x_2}(x-x_3)+\Phi(x_3).
       \end{equation}

       If there is an $a\in (x_1,x_2)$ such that  
       $\Phi(a)< \frac{\Phi(x_3)-\Phi(x_2)}{x_3-x_2}(a-x_3)+\Phi(x_3) $, then  $\frac{\Phi(x_3)-\Phi(a)}{x_3-a}>\frac{\Phi(x_3)-\Phi(x_2)}{x_3-x_2}$, and  we have that $\Phi(x_2)>\frac{\Phi(x_3)-\Phi(a)}{x_3-a}(x_2-x_3)+\Phi(x_3)$. This inequality means that $x_2$ is a convex combination between $a$ and $x_3$ such that $\Phi(x_2)$ is not in the segment between the points $(a,\Phi(a))$ and $(x_3, \Phi(x_3))$; which contradicts the convexity of $\Phi$. A similar argument is used if $a\in (x_2,x_3)$. So (\ref{eq de la cuerda cuerda}) is satisfied for all $x$ in $J$.
\end{proof}

\section{Uniqueness Results}

Next we give  uniqueness results of best $\Phi-$ approximation, for a suitable convex function $\Phi,$ which is non strictly convex and, according to Lemma \ref{lema no estrict conv}, it is a straight line in some interval. In each case a convenient approximation class is considered.

\begin{thm} 
    Let $\Phi\in \Psi,$ $S_n\subseteq C([a,b])$ be a Tchebycheff space in $[a,b]$ and $f$ be a continuous function in $[a,b]$. Then there exists a unique best $\Phi$-approximation of $f$ in $[a,b]$ from the class $S_n$.
\end{thm}
\begin{proof}
    Since the approximation class is finite dimensional, the existence of a best approximation, say $P_1$, is assured by theorem 2.3 in \cite{BFL}.  It remains to prove the uniqueness of the best $\Phi-$ approximation function. Suppose that $f\not \in S_n$ and for each $Q\in S_n$, consider the set $Z(f-Q):=\set{x\in [a,b] ~:~ f(x)-Q(x)=0}$. 
    
    Suppose that $Z(f-P_1)$  has measure zero, then by Theorem \ref{thm: caracterización} we conclude \begin{align*}\int_{\{Q>0,~ f>P_1\}}\varphi^{-}(|f-P_1|) sgn(f-P_{1})Q \,dx + \int_{\{Q<0,~ f<P_1\}}\varphi^{-}(|f-P_1|) sgn(f-P_{1})Q \,dx \\ +  \int_{\{Q<0, ~ f>P_1 \}} \varphi^{+}(|f-P_1|) sgn(f-P_{1})Q \,dx + \int_{\{Q>0, ~ f<P_1 \}} \varphi^{+}(|f-P_1|) sgn(f-P_{1})Q \,dx \\ \leq 0 \end{align*} for every $Q\in S_{n}$.   
    
    Note that $S_n$ has a basis $\set{\delta_i}_{i=1}^n$ such that for $m<n$, the set  $\set{\delta_i}_{i=1}^{m}$ generate also a Tchebycheff space (see Theorem 2.29 in \cite{Schumaker}). 
    
    We get that the function $sign[f(x)-P_1(x)]$ has at least $n$ sign changes. In fact, suppose it has  only $m<n$ sign changes at the $n$ points $x_1<...<x_m$ in $(a,b)$.   Then 
    there exists a  function $P_m\in S_n,\,\, P_m\neq 0$  that changes sign only  in  $x_i$, for $i=1,...,m$ (see Proposition 2, page 195, in \cite{PinkusLibro}). It follows that the function $P_m sign[f-P_1] \ge 0$ in $[a,b].$

 Since $\varphi^{+}$ and $\varphi^{-}$ are positive for $x>0$, we have \begin{align*}0=\int_{\{P_{m}>0,~ f>P_1\}}\varphi^{-}(|f-P_1|) sgn(f-P_{1})P_{m} \,dx \\ = \int_{\{P_{m}<0,~ f<P_1\}}\varphi^{-}(|f-P_1|) sgn(f-P_{1})P_{m} \,dx \\ =  \int_{\{P_{m}<0, ~ f>P_1 \}} \varphi^{+}(|f-P_1|) sgn(f-P_{1})P_{m} \,dx \\ = \int_{\{P_{m}>0, ~ f<P_1 \}} \varphi^{+}(|f-P_1|) sgn(f-P_{1})P_{m} \,dx \end{align*}

Thus we conclude  $f=P_1$ almost everywhere in $[a,b],$ we get a contradiction. 
 
 Then, the function $sgn[f(x)-P_1(x)]$, has at least $n$ sign changes.
    Now, suppose that there exists $P_2$, another best $\Phi$-approximation for $f$ in $[a,b]$. By Lemma \ref{lema 2 mejores app}, $f(x)-P_2(x)$ is zero at the $n-$points where $f(x)-P_1(x)$ changes signs. Then $P_1(x)-P_2(x)$ has $n-$zeros, so $P_1=P_2$.

    Now, suppose that for any $P_1\in \mu_{\Phi}(f/S_n)$ we have $m(Z(f-P_1)) >0$ and consider another $P_2\in \mu_{\Phi}(f/S_n)$. By the convexity of $\Phi$, we have for $\lambda\in[0,1]$: $$\int_{[a,b]} \Phi(|f-\lambda P_1-(1-\lambda)P_2|)dx\leq \lambda\int_{[a,b]} \Phi(|f-P_1|)dx+(1-\lambda)\int_{[a,b]}\Phi(|f-P_2|)dx$$ then $\lambda P_1 + (1-\lambda)P_2 \in \mu_{\Phi}(f/S_n)$. Suppose that for each $\lambda\in [0,1]$: $m(Z(f-\lambda P_1 - (1-\lambda)P_2))>0$. Then there exists $\lambda_1,~ \lambda_2\in [0,1]$ such that the sets of zeros $Z(f-\lambda_1 P_1 - (1-\lambda_1)P_2)$ and $ Z(f-\lambda_2 P_1 - (1-\lambda_2)P_2)$ have intersection with positive measure (see Lemma 4-7, page 109 in \cite{Rice}). In particular, the intersection has $n-$points. Then $P_1=P_2$ and this concludes the proof.
    \end{proof}
    Next we consider  wider approximation classes: the {\em 1-space} and the {\em 0-space}.
    \begin{defi} 
    We say that the n-dimensional set $S_n\subset C([a,b])$  is a {\em 1-space} if there exists $h\in S_n$ such that $h(x) > 0$ for every $x\in [a,b],$
    and if $P_1, P_2 \in S_n$ with $P_1 \neq P_2$ then $|\{P_1 = P_2\}| =0.$ If $S_n$ satisfies only the last condition, we say that $S_n$ is a {\em 0-space}.
    \end{defi}

    \begin{lem}\label{lema14}
        Let $S_n$ a 1-space and $P\in\mu_{\Phi}(f/S_n).$ Then there exits $x\in [a,b]$ such that $f(x) = P(x).$
    \end{lem}
    \begin{proof}
    We assume $P(x) - f(x) > 0,$ for every $x\in [a,b]$ and consider $h\in S_n, \,\, h(x)> 0$ for $x\in [a,b].$ Then, for $\varepsilon >0$ such that $\varepsilon < \frac{min_{x\in [a,b]} P(x) - f(x)}{\max_{x\in [a,b]} h(x)},$ we have $$ \varepsilon h(x) < P(x) - f(x), \,\, x\in [a,b].$$
    Then $$0 < \int_{[a,b]} \Phi (P - \varepsilon h -f)\, dx < \int_{[a,b]} \Phi(P-f) \, dx,$$
    which is a contradiction since $P - \varepsilon h \in S_n$
    \end{proof}
    Now we get the following uniqueness result when $S_n$ is a 1-space.
    \begin{thm}
     Let $\varphi \in \Im$ which is also a strictly increasing function in $[0,b], \, b>0,$ and $\Phi (x) = \int_0^x \varphi (t) \, dt.$ If $S_n$ is a 1-space, then for every $f\in C([a,b])$ the best approximation set $\mu_{\Phi} (f/S_n)$ is a singleton. 
        \end{thm}
        \begin{proof}
            Suppose $P_1, P_2 \in \mu_{\Phi} (f/{S_n}),$ with $P_1 \neq P_2.$ Now by Lemma \ref{lema14} there exists an interval $I$ such that $|f-P_1|(x) \le b,$ for every $x\in I, $ and then the set $J= I \cap \{x\in [a,b] : P_1 (x) \neq P_2 (x)\}$ has positive measure. Then, from the strictly convexity of $\Phi$ in $[0, b]$ and taking into account Lemma \ref{lema 2 mejores app} we get
            $$\int_{J} \Phi (|f-(\frac{P_1 + P_2}{2})|) \, dx< \frac{1}{2} \int_{J} \Phi(|f-P_1|)\, dx + \frac{1}{2} \int_{J} \Phi(|f-P_2|)\, dx$$ and then, since the convexity of $\Phi,$  we obtain
            $$\int_{[a,b]} \Phi (|f-(\frac{P_1 + P_2}{2})|) \, dx< \frac{1}{2} \int_{[a,b]} \Phi(|f-P_1|)\, dx + \frac{1}{2} \int_{[a,b]} \Phi(|f-P_2|)\, dx,$$
            which is a contradiction since $\frac{P_1 + P_2}{2} \in S_n.$
         \end{proof}
        
    For the next result we consider a n-dimensional 0-space $S_n \subseteq C([a,b])$ and we give a necessary and sufficient condition  on the approximation class $S_n$ which assures the uniqueness of $\mu_{\Phi} (f/S_n)$ for convex function $\Phi,$
    which is a strictly convex function just for $x$ grater than some positive real number. This result follows the same lines considered in \cite{CW} for  $L^1$  and it is a generalization  to Orlicz spaces. 

     \begin{defi}\label{betaset} 
       The set $Z(f):=\{x\in [a,b]: f(x)=0\}$ is called a $\gamma-$set if  $0\in\mu_{\Phi}(f/S_n)$.  
    \end{defi}

   \begin{thm}\label{unicitybeta}
   %
   
   Consider the function $\Phi$ such that for $k, c >0$ and $x\in [0,c]$,  $\Phi(x) = k x,$ and $\Phi$ is a strictly convex function for $x\in (c, \infty)$. Let $S_n$ be a n-dimensional 0-space. Then, for every $f\in C([a,b]) \cap L^{\Phi} ([a,b])$ the set $\mu_{\Phi} (f/S_n)$ has an unique element $P_1,$  if and only if
   \begin{enumerate}[a)]
   \item The constant function $0$ is the unique element of $S_n$ which is $0$ on a $\gamma-$ set \\   
   or 
   
   \item For $P_1 \in \mu_{\Phi} (f/S_n), $ the set $\{x\in [a,b] : |f(x)-P_1(x)| > c\}$ has positive measure for every  $f\in C([a,b]) \cap L^{\Phi} ([a,b]).$
   \end{enumerate}
   \end{thm}

   \begin{proof}
   Suppose $P_1, P_2$ are two different elements in $\mu_{\Phi} (f/S_n).$ Due to the convexity of $\Phi$, we have 
   $$\int_{[a,b]} \Phi (|f- \frac{P_1 + P_2}{2} |) \, dx = \frac{1}{2} \int_{[a,b]} \Phi (|f- P_1 |)\, dx +  \frac{1}{2} \int_{[a,b]} \Phi (|f- P_2 |)\, dx.$$
   Then we get 
   $$ \Phi (|f(x)- \frac{P_1 + P_2}{2} (x)|)  = \frac{1}{2} \Phi (|f(x)- P_1 (x)|) +  \frac{1}{2}  \Phi (|f(x)- P_2 (x)|),\,\,\,  x\in [a,b].$$
   Thus for any $x\in Z(f-\frac{P_1 +P_2}{2})$ we have $f(x)- P_1 (x) = f(x)- P_2(x),$ and then $P_1 (x) - P_2 (x) = 0, $ and this is a contradiction of a).

   Now assuming condition b) and suppose  $P_1,\, P_2 \in \mu_{\Phi} (f/S_n),\,\, P_2 \neq P_1,$ then we have $m(I)>0,$ for 
   $$I =\{x\in [a,b] : |f(x)-P_1 (x)|> c\} \cap \{x\in [a,b] : P_1 (x) \neq P_2 (x)\},$$ thus 
   $$\int_I \Phi (|f-(\frac{P_1+P_2}{2})|)\, dx < \frac{1}{2} \int_I \Phi (|f-P_1|) \, dx+ \frac{1}{2} \int_I \Phi (|f-P_2|)\, dx,$$
   and then we have a contradiction. 
   
   On the other hand, suppose there exists $P_3 \in S_n, \,\, P_3 \neq 0$ and $0\le |P_3| \le \frac{c}{2},$ such that $Z(f-P_1) \subseteq Z(P_3), \, P_1 \in \mu_{\Phi} (f/S_n)$  and assume
   $0\le |f-P_1|(x) \le c,$ a. e. $x$. 

   Now,  for $h (x) = |P_3 (x)| sgn (f-P_1) (x),$ we have 
   $$\int_{[a,b] \cap \{f\neq P_1\}} \varphi (|h|) sgn (h) Q\, dx = \int_{[a,b] \cap \{f\neq P_1\}} \varphi (|h|) sgn (f-P_1) Q\, dx$$
   and since $\varphi(|h|) = \varphi(|f-P_1|) = k,$ 
   we have 
   $$\int_{[a,b] \cap \{f\neq P_1\}} \varphi (|h|) sgn (h) Q\, dx = \int_{[a,b] \cap \{f\neq P_1\}} \varphi (|f-P_1|) sgn (f-P_1) Q\, dx,$$
   and since $P_1 \in \mu_{\Phi} (f/S_n)$ and $Z(f-P_1) \subseteq Z(h),$ we get
   $$\int_{[a,b] \cap \{h\neq 0\}} \varphi (|h|) \,sgn (h) \,Q \, dx\le \varphi(0) \int_{Z(h)} |Q|\, dx,$$
   which implies, since \ref{remark caracterizacion} that $0\in \mu_{\Phi} (h/S_n).$

   Now 
   $$\int_{[a,b]} \Phi (|h- \varepsilon P_3|) \, dx= \int_{[a,b]} \Phi ((h- \varepsilon P_3) \,sgn (h- \varepsilon P_3)\, dx,$$
   and taking into account $\Phi (x) = k x,$ for $x\le c$ we get
   $$\int_{[a,b]} \Phi (|h- \varepsilon P_3|) \, dx= \int_{[a,b]} k\, h \, sgn(h-\varepsilon P_3) \, dx- \varepsilon \int_{[a,b]} k\, P_3\, sgn(h-\varepsilon P_3)\, dx, $$
   and then, for $x\in [a,b],\,\, h(x)-\varepsilon P_3 (x) \neq 0,$ it holds $sgn (h- \varepsilon P_3) (x) = sgn (f-P_1) (x),$
   thus
   $$\int_{[a,b]} \Phi (|h- \varepsilon P_3|)\, dx = \int_{[a,b]} k\, h\, sgn(f-P_1)\, dx - \varepsilon \int_{[a,b]} k\, P_3\, sgn(f-P_1)\, dx.$$
   Now, we use remark \ref{remark caracterizacion} to obtain 
   $$- \varepsilon \int_{[a,b]} k\, P_3\,  sgn(f-P_1)\, dx \le \varepsilon \varphi (0) \int_{Z(f-P_1)} k\, |P_3|\, dx,$$
   and then
   $$\int_{[a,b]} \Phi (|h- \varepsilon P_3|) \, dx\le \int_{[a,b]} k\, h\, sgn(f-P_1)\, dx + \varepsilon \varphi (0) \int_{Z(f-P_1)} k\, |P_3|\, dx.$$
   Thus, since $Z(f-P_1) \subseteq Z(P_3),$ we have $\int_{Z(f-P_1)}  |P_3| \, dx=0,$ then
   $$\int_{[a,b]} \Phi (|h- \varepsilon P_3|)\, dx \le \int_{[a,b]} k\, h\, sgn(f-P_1)\, dx = \int_{[a,b]} k\, h\, sgn(h) \, dx= \int_{[a,b]} \Phi (|h|)\, dx, $$
   which implies that $\varepsilon P_3 \in \mu_{\Phi} (h/S_n),$ for every $\varepsilon, \,\, 0< \varepsilon <1.$
     
   \end{proof}

   Finally, we set an uniqueness result of $\mu_{\Phi} (f/S_n),$ where a specific convex function $\Phi$ and a suitable 1-space  $S_n$ are considered. 

   \begin{thm}
   Let $\Phi \in \Psi$, $\Phi(x) = \int_0^x \varphi (t) \, dt,$ where $\varphi \in \Im$ is non continuous in a decreasing positive sequence $a_n,$ which converges to $0.$ Let $S_n$ be a 1-space such that every nonzero $P\in S_n$ has only a finite amount of zeros on $[a,b]$. Then for every $f\in C([a,b])$ the best approximation set $\mu_{\Phi} (f/S_n)$ is a singleton. 
   \end{thm}
   \begin{proof}
   Suppose that $P_1, P_2\in\mu_{\Phi} (f/S_{n})$ with $P_1 \neq P_2$. Since
 $\int_{a}^{b}\Phi(|f-P_1|)-\Phi(|f-P_2|)dx=0$,
 the function $\Phi(|f-P_1|)-\Phi(|f-P_2|)$ cannot be  positive in $[a,b]$.
Then, since $\Phi(0)=0$ and  Lemma \ref{lema 2 mejores app}, there exists $x_0\in[a,b]$ such that $f(x)-P_1(x_0)=f(x)-P_2(x_0)$. 
 
Actually, $x_0$ is a zero of $P_{1}-P_2$, which is a nonzero element of $S_n$. Therefore,  $x_0$ is an isolated point in $[a,b]$ and there exists a positive real number
  $\delta_{0}>0$ such that for every $x\in(x_{0},x_{0}+\delta_{0}]$, we assume without loss of generality that $f(x)-P_1(x)>f(x)-P_{2}(x)>0$.
 
 Also there exists a real number $x_{1}\in[x_{0},x_{0}+\delta_{0}]$ such that $$(f-P_1)(x_{1})=\max_{x\in[x_{0},x_{0}+\delta_{0}]}(f-P_1)(x)$$
 
 Considering that ${a_n}\rightarrow 0$, there must be some $n$ such that $a_{n}<(f-P_1)(x_{1})$. Due to $f-P_1$ continuity,  there is a real number $x_{2}\in[x_{0},x_{0}+\delta_{0}]$ such that $(f-P_1)(x_{2})= a_{n}$.
 
 For this $x_{2}$ we can find some $\delta_{1}>0$ such that $(f-P_{2})(x)<a_{n}<(f-P_1)(x)$ for every $x\in(x_{2},x_{2}+\delta_{1})\subset[x_{0},x_{0}+\delta_{0}]$.
 If there were no such  $\delta_{1}>0$, then $f-P_1$ would not reach up to its maximum at $(f-P_1)(x_{1})$.
 
 Finally, if we set $I:=(x_{2},x_{2}+\delta_{1})$, we can see that
 \begin{eqnarray*}
 \int_{[a,b]} \Phi(|f-(\frac{P_1+P_{2}}{2})|)\, dx &=&\int_{I} \Phi(|f-(\frac{P_1+P_{2}}{2})|)\, dx+\int_{[a,b]-I} \Phi(|f-(\frac{P_1+P_{2}}{2})|)\, dx\\
                                              & < & \frac{1}{2}\int_{I}\Phi(|f-P_1|)\, dx+\frac{1}{2}\int_{I}\Phi(|f-P_{2}|)\, dx+\\
                                              &~~ &  +\frac{1}{2}\int_{[a,b]-I}\Phi(|f-P_1|)\, dx+\frac{1}{2}\int_{[a,b]-I}\Phi(|f-P_{2}|)\, dx\\
                                              &=& \frac{1}{2}\int_{[a,b]}\Phi(|f-P_1|)\, dx+\frac{1}{2}\int_{[a,b]}\Phi(|f-P_{2}|)\, dx.
 \end{eqnarray*}

This is a contradiction since $\frac{P_1+P_{2}}{2}\in S_n$.	
   \end{proof}



\bibliographystyle{amsplain}

\end{document}